\documentclass{amsart}
\usepackage{amssymb,amsmath}
\usepackage{graphicx,subfigure}
\newtheorem{theorem}{Theorem}%[section]
\newtheorem{proposition}[theorem]{Proposition}

\newtheorem{lemma}[theorem]{Lemma}

\theoremstyle{definition}

\theoremstyle{remark}

\newcommand{\de}{\delta}
\newcommand{\ep}{\epsilon}

\newcommand{\om}{\omega}
\newcommand{\si}{\sigma}
\newcommand{\te}{\theta}
\newcommand{\vp}{\varphi}

\newcommand{\De}{\Delta}
\newcommand{\Ga}{\Gamma}

\newcommand{\Om}{\Omega}

\newcommand{\BOm}{\overline\Omega}

\newcommand{\cU}{{\mathcal U}}
\newcommand{\cK}{{\mathcal K}}

\newcommand{\ty}{\widetilde{y}}

\newcommand{\tK}{\widetilde{K}}

\newcommand{\tom}{\widetilde{\om}}

\def\RR{\mathbb{R}}

\renewcommand\SS{\mathbb{S}}
\newcommand{\cB}{{\mathcal B}}

\newcommand{\cE}{{\mathcal E}}

\newcommand{\cI}{{\mathcal I}}
\newcommand{\cJ}{{\mathcal J}}

\newcommand{\cO}{{\mathcal O}}

\newcommand{\cS}{{\mathcal S}}

\def\th{^{\text{th}}}
\def\bw{{\bar w}}
\newcommand{\pd}{\partial}
\newcommand\minus\backslash

\newcommand\lan\langle
\newcommand\ran\rangle

\def\Div{\nabla\cdot }

\DeclareMathOperator\dist{dist}

\renewcommand\leq\leqslant
\renewcommand\geq\geqslant
\newlength{\intwidth}

\makeatletter
\newcommand*{\mint}[1]{%
  \mint@l{#1}{}%
}
\newcommand*{\mint@l}[2]{%
  \@ifnextchar\limits{%
    \mint@l{#1}%
  }{%
    \@ifnextchar\nolimits{%
      \mint@l{#1}%
    }{%
      \@ifnextchar\displaylimits{%
        \mint@l{#1}%
      }{%
        \mint@s{#2}{#1}%
      }%
    }%
  }%
}
\newcommand*{\mint@s}[2]{%
  \@ifnextchar_{%
    \mint@sub{#1}{#2}%
  }{%
    \@ifnextchar^{%
      \mint@sup{#1}{#2}%
    }{%
      \mint@{#1}{#2}{}{}%
    }%
  }%
}
\def\mint@sub#1#2_#3{%
  \@ifnextchar^{%
    \mint@sub@sup{#1}{#2}{#3}%
  }{%
    \mint@{#1}{#2}{#3}{}%
  }%
}
\def\mint@sup#1#2^#3{%
  \@ifnextchar_{%
    \mint@sub@sup{#1}{#2}{#3}%
  }{%
    \mint@{#1}{#2}{}{#3}%
  }%
}
\def\mint@sub@sup#1#2#3^#4{%
  \mint@{#1}{#2}{#3}{#4}%
}
\def\mint@sup@sub#1#2#3_#4{%
  \mint@{#1}{#2}{#4}{#3}%
}
\newcommand*{\mint@}[4]{%
  \mathop{}%
  \mkern-\thinmuskip
  \mathchoice{%
    \mint@@{#1}{#2}{#3}{#4}%
        \displaystyle\textstyle\scriptstyle
  }{%
    \mint@@{#1}{#2}{#3}{#4}%
        \textstyle\scriptstyle\scriptstyle
  }{%
    \mint@@{#1}{#2}{#3}{#4}%
        \scriptstyle\scriptscriptstyle\scriptscriptstyle
  }{%
    \mint@@{#1}{#2}{#3}{#4}%
        \scriptscriptstyle\scriptscriptstyle\scriptscriptstyle
  }%
  \mkern-\thinmuskip
  \int#1%
  \ifx\\#3\\\else_{#3}\fi
  \ifx\\#4\\\else^{#4}\fi  
}
\newcommand*{\mint@@}[7]{%
  \begingroup
    \sbox0{$#5\int\m@th$}%
    \sbox2{$#5\int_{}\m@th$}%
    \dimen2=\wd0 %
    \let\mint@limits=#1\relax
    \ifx\mint@limits\relax
      \sbox4{$#5\int_{\kern1sp}^{\kern1sp}\m@th$}%
      \ifdim\wd4>\wd2 %
        \let\mint@limits=\nolimits
      \else
        \let\mint@limits=\limits
      \fi
    \fi
    \ifx\mint@limits\displaylimits
      \ifx#5\displaystyle
        \let\mint@limits=\limits
      \fi
    \fi
    \ifx\mint@limits\limits
      \sbox0{$#7#3\m@th$}%
      \sbox2{$#7#4\m@th$}%
      \ifdim\wd0>\dimen2 %
        \dimen2=\wd0 %
      \fi
      \ifdim\wd2>\dimen2 %
        \dimen2=\wd2 %
      \fi
    \fi
    \rlap{%
      $#5%
        \vcenter{%
          \hbox to\dimen2{%
            \hss
            $#6{#2}\m@th$%
            \hss
          }%
        }%
      $%
    }%
  \endgroup
}

\newcommand\loc{_{\mathrm{loc}}}

\DeclareMathOperator\PV{PV}

 \def\curl{\nabla\times }

\DeclareMathOperator\BS{BS}

\def\hot{\text{h.o.t.}}

\begin{document}

\title{The Biot--Savart operator of a bounded domain}

\author{Alberto Enciso}
\address{Instituto de Ciencias Matem\'aticas, Consejo Superior de
  Investigaciones Cient\'\i ficas, 28049 Madrid, Spain}
\email{aenciso@icmat.es, mag.ferrero@icmat.es, dperalta@icmat.es}

\author{M.\ \'Angeles Garc\'\i a-Ferrero}
%\address{Instituto de Ciencias Matem\'aticas, Consejo Superior de
%  Investigaciones Cient\'\i ficas, 28049 Madrid, Spain}

\author{Daniel Peralta-Salas}
%\address{Instituto de Ciencias Matem\'aticas, Consejo Superior de
%  Investigaciones Cient\'\i ficas, 28049 Madrid, Spain}
%\email{dperalta@icmat.es}

%%    General info
%\subjclass[2010]{35B38, 58J05, 58K45}
%\date{\today}
%
%\keywords{ }
%
\begin{abstract}
We construct the analog of the Biot--Savart integral for bounded
domains. Specifically, we show that the velocity field of an incompressible fluid with
tangency boundary conditions on a bounded domain can be
written in terms of its vorticity using an integral kernel
$K_\Om(x,y)$ that has an
inverse-square singularity on the diagonal. 
\end{abstract}
\maketitle

\section{Introduction}

The Biot--Savart operator,
\[
\BS(\om)(x):=\int_{\RR^3} \frac{\om(y)\times (x-y)}{4\pi|x-y|^3}\, dy\,,
\]
plays a key role in fluid mechanics and electromagnetism as a sort of
inverse of the curl operator. More precisely, if $\om$ is a
well-behaved divergence-free vector field on~$\RR^3$, then
$u:=\BS(\om)$ is the only solution to the equation
\[
\curl u=\om\,,\qquad \Div u=0
\]
that falls off at infinity. Consequently, in fluid mechanics the Biot--Savart
operator maps the vorticity~$\om$ of a fluid into its associated velocity field~$u$.

In this paper we will be concerned with the problem of mapping the
vorticity of a fluid contained in a bounded domain~$\Om$ of $\RR^3$ into its
velocity field. To put it differently, given a vector field $\om$ we want to solve the problem
\begin{equation}\label{problem}
\curl u= \om\,, \qquad \Div u=0\,, \qquad u\cdot \nu =0
\end{equation}
in~$\Om$, where the tangency condition $u\cdot \nu=0$ means that the
fluid stays inside the domain. We will assume throughout that the boundary of this domain is smooth,
although one could relax this condition. 

It is easy to see that, for this equation to admit a
solution, the field~$\om$ must satisfy several hypotheses. Firstly,
since the divergence of a curl is zero, it is obvious that $\om$ must
be divergence-free. Secondly, it is easy to see that
if $\Ga_1,\dots,\Ga_m$ denote the connected components of~$\pd\Om$,
then one must have
\begin{equation}\label{hypothesis}
\int_{\Ga_j}\om\cdot\nu\, d\si=0\qquad \text{for all } 1\leq j\leq m\,.
\end{equation}
To see why this is true, it is enough to take the harmonic functions $\psi_j$
defined by the boundary value problems
\[
\De\psi_j=0 \quad \text{in }\Om\,,\qquad \psi_j|_{\Ga_k}=\de_{jk}
\]
and observe that
\[
\int_{\Ga_j}\om\cdot\nu \, d\si= \int_{\Om}\Div(\psi_j\om)\, dx=
\int_\Om \nabla\psi_j\cdot\curl u\, dx=\int_{\pd\Om}\, u\cdot
(\nabla\psi_j\times \nu)\, d\si=0\,.
\]
In addition to using that $\om$ is divergence-free and integrating
by parts, we have exploited that the gradient $\nabla\psi_j$ is proportional to
$\nu$ at the boundary.

The problem~\eqref{problem} has been considered by a number of
people, who have shown the existence of solutions for
domains of different regularity and derived estimates in $L^p$ or
H\"older spaces. Up to date accounts of the problem can be found e.g.\ in~\cite{Amrouche,Amrouche2,Shkoller} and
references therein. However, the question of whether the solution is
given by an integral formula generalizing the classical Biot--Savart
law remains wide open.

Our objective in this paper is to fill this gap by showing that one can construct a
solution to the problem~\eqref{problem} through a generalized
Biot--Savart operator with a reasonably
well-behaved integral kernel. We will only state our result for the
flat-space problem, but it will be apparent from the construction that
the result holds true (mutatis mutandis) for bounded domains in any
Riemannian 3-manifold. The existence of a Biot--Savart
operator on a compact Riemannian 3-manifold without boundary can be obtained from the Green's function of
the Hodge Laplacian, computed in~\cite{deRham}, although to the best
of our knowledge the only result available in the literature~\cite{Vogel} is a
weaker analog of the Biot--Savart operator on a closed manifold that
provides a solution to the equation up to a gradient field (i.e., a
vector field satisfying $\curl u=\om+\nabla\vp$, where $\vp$ is some
smooth function). It is should be stressed that this connection
between the Green's function of the Laplacian in a domain (with
certain boundary conditions) and the Biot--Savart operator is
no longer true in the presence of a boundary condition. Indeed, the
combination of the boundary condition with the fact that the field
must be divergence-free makes the structure of the generalized
Biot--Savart kernel rather involved.

Our motivation for this is threefold. Firstly, the integral kernel of
the Biot--Savart operator on a compact 3-manifold without boundary has
been recently employed~\cite{PNAS} to show that the helicity is the
only regular integral invariant of volume-preserving
transformations. Proving a similar result for the case of manifolds with boundary presents
additional difficulties, but in any case it is important to have a
thorough understanding of the associated Biot--Savart
operator. A particular case has been recently established
in~\cite{Kudryavtseva}. Applications to electrodynamics of the Biot--Savart operator for
domains in the 3-sphere (whose existence as an integral operator is
not discussed, however) can be found in~\cite{Parsley}.  Secondly, the
existence of a generalized Biot--Savart integral in domains ensures
that the celebrated connection between the helicity of a field and the
linking number, unveiled by Arnold in the full space~\cite{Arnold},
remains valid in bounded domains.
A third motivation to construct the Biot--Savart operator in domains
is that the structure of the integral kernels of the
inverses of operators has been key to develop certain approximation
theorems that we have exploited in different
contexts~\cite{Annals,Adv,Acta}.

To state the existence of a Biot--Savart integral
operator on~$\Om$, let us define the function
\begin{equation}\label{defl}
\ell(y):=\log\bigg(2+\frac1{\dist(y,\pd\Om)}\bigg)\,.
\end{equation}
Let us also recall that a vector field $h$ on~$\Om$ is said to be
harmonic and tangent to the boundary if
\[
\curl h=0\,,\quad \Div h=0\,,\quad h\cdot\nu=0\,.
\]
By Hodge theory, the dimension of
the linear space of tangent harmonic fields is the genus
of~$\pd\Om$ (if $\pd\Om$ is disconnected, this is defined as the sum
of the genus of the connected components of the boundary). The main
result of the paper can then be presented as follows:

\begin{theorem}\label{T.main}
Let $\om\in W^{k,p}(\Om)$ be a divergence-free vector field satisfying
the hypothesis~\eqref{hypothesis}, with $k\geq0$ and $1<p<\infty$. Then the boundary-value
problem~\eqref{problem} has a solution $u\in W^{k+1,p}(\Om)$
that satisfies the estimate
\begin{equation}\label{estimate}
\|u\|_{W^{k+1,p}(\Om)}\leq C\|\om\|_{W^{k,p}(\Om)}
\end{equation}
and which can be represented as an integral of the
form
\[
u(x)=\BS_\Om(\om)(x):=\int_\Om K_\Om(x,y)\, \om(y)\, dy\,,
\]
where $K_\Om(x,y)$ is a matrix-valued integral kernel that is smooth
outside the diagonal and
satisfies the pointwise bound
\[
|K_\Om(x,y)|\leq C\frac {\ell(y)}{|x-y|^2}\,.
\]
Furthermore, the solution is unique modulo the
addition of a harmonic field tangent to the boundary. 
\end{theorem}

Let us emphasize that the core of the result is not the
estimate~\eqref{estimate}, which is not new, but the existence of an integral kernel
$K_\Om(x,y)$ with the above properties. Of course, a serious difficulty that arises in the analysis of the integral
kernel $K_\Om(x,y)$ of the Biot--Savart operator on a domain is that
it is strongly non-unique. Indeed, it is easy to see that if $\phi$ is
a scalar function on $\Om$ satisfying the Dirichlet boundary condition
$\phi|_{\pd\Om}=0$, then its gradient is orthogonal to any
divergence-free field~$\om$ in the sense that
\[
\int_\Om \nabla \phi\cdot \om\, dx=0
\]
for any vector field with $\Div \om=0$. Hence if $K'(x,y)$ is a matrix-valued
function of the form
\[
K'_{ij}(x,y)=\pd_{y_j} A_i(x,y)
\]
with $A_i(x,\cdot)|_{\pd\Om}=0$, then $K_\Om(x,y)+K'(x,y)$ is also an admissible
kernel for the Biot--Savart operator of the domain. 

The paper is organized as follows. In Section~\ref{S.BS} we will
present some estimates and identities for the usual Biot--Savart
integral that are needed later. In Section~\ref{S.extension} we
construct an extension operator for vector fields we employ during the
construction of the kernel $K_\Om(x,y)$ to deal with boundary
terms. In Section~\ref{S.construction} we construct the solution~$u$
to the problem~\eqref{problem} using layer potentials. This is
convenient because it leads to quite explicit formulas, which we
carefully analyze in Section~\ref{S.kernel} to establish the existence
of the desired integral kernel. The proof works without any major
modifications on any Riemannian 3-manifold, using the layer
potentials and the Biot--Savart operator associated with the
metric. To conclude, for completeness we discuss in
Section~\ref{S.uniqueness} the uniqueness of the solutions (possibly
with nonzero but prescribed divergence and normal component on the
boundary).

\section{Estimates and identities for the Biot--Savart integral}
\label{S.BS}

This section is a brief but reasonably self-contained presentation of several results and identities for the
Biot--Savart integral that will be used later. These results
are essentially standard. 

Given a vector field $F\in C^1(\BOm)$, we will state the results in
terms of the field
\[
w(x):=\int_\Om \frac{F(y) \times(x-y)}{4\pi|x-y|^3}\, dy\,.
\]
In what follows, $\ep_{ijk}$ will denote Levi-Civita's permutation
symbol and $B_r(x)$ (resp.~$B_r$) will denote the three-dimensional
ball of radius~$r$ centered at the point~$x$ (resp.\ at the origin).

\begin{lemma}\label{L.Dw}
The derivative of the field $w$ at any
point $x\in\Om$ is
\begin{multline}\label{Dwbien}
  \pd_jw_k(x) = \ep_{klm}\,\PV \int_\Om F_l(y)\,
  \frac{|x-y|^2\de_{jm}-3(x_j-y_j)(x_m-y_m)}{4\pi|x-y|^5} \, dy \\
  - \ep_{klm} F_l(x)\, A^j_m(x)\,,
\end{multline}
where $A^j _m$ stands for the $m\th$ component of a certain continuous
vector field $A^j \in L^\infty(\Om)$ and $\PV$ denotes the principal value.
\end{lemma}

\begin{proof}
Let us take a smooth function $\eta(r)$ which vanishes for $r<\frac12$ and
is equal to~1 for $r>1$ and set $\eta_\de(r):=\eta(r/\de)$, with $\de$
a small positive constant. Then let
us define the vector field
\[
w^\de(x):= \int_\Om \eta_\de(|x-y|)\,\frac{F(y) \times(x-y)}{4\pi|x-y|^3}\, dy\,.
\]
It is apparent that $w^\de$ is a smooth vector field
on~$\RR^3$. Moreover, it is not hard to see that $w^\de$ converges
to~$w$ uniformly, since for any~$x\in\Om$ one has
\begin{align*}
|w(x)-w^\de(x)|&=\bigg|\int_{B_\de(x)}(1- \eta_\de(|x-y|))\,\frac{F(y)
  \times(x-y)}{4\pi|x-y|^3}\, dy\bigg|\\
&\leq
(1+\|\eta\|_{L^\infty})\|F\|_{L^\infty(\Om)}\int_{B_\de}\frac{dz}{4\pi|z|^2}\\
&\leq C\de\,.
\end{align*}
The derivative of the $k\th$ component of $w^\de$ can be readily
computed as
\begin{align}
\pd_j w^\de_k(x)& =\ep_{klm}\int_\Om F_l(y)\, \pd_{x_j}\bigg(\eta_\de(|x-y|)\,
\frac{x_m-y_m}{4\pi|x-y|^3}\bigg) \, dy\notag\\
&= \ep_{klm}\int_\Om [F_l(y)-F_l(x)]\, \pd_{x_j}\bigg(\eta_\de(|x-y|)\,
\frac{x_m-y_m}{4\pi|x-y|^3}\bigg) \, dy\notag \\
&\qquad \qquad\qquad   - \ep_{klm} F_l(x)\int_\Om \pd_{y_j}\bigg(\eta_\de(|x-y|)\,
\frac{x_m-y_m}{4\pi|x-y|^3}\bigg) \, dy\notag\\
&= \ep_{klm}\int_\Om [F_l(y)-F_l(x)]\, \pd_{x_j}\bigg(\eta_\de(|x-y|)\,
\frac{x_m-y_m}{4\pi|x-y|^3}\bigg) \, dy \label{Dw}\\%\label{Dw2}
&\qquad \qquad \qquad \qquad \qquad \qquad \qquad \qquad\qquad - \ep_{klm} F_l(x)\, A_m^{j,\de}(x)\,,\notag
\end{align}
where for each $x\in\Om$ we have set
\[
A_m^{j,\de}(x):= \int_{\pd\Om} \nu_j(y)\,\eta_\de(|x-y|)\,
\frac{x_m-y_m}{4\pi|x-y|^3} \, d\si(y)\,.
\]

As $\de\to0$, the first term converges to the
principal value integral of the statement, since the difference
\begin{multline*}
M:=\int_\Om [F_l(y)-F_l(x)]\, \pd_{x_j}\bigg(\eta_\de(|x-y|)\,
\frac{x_m-y_m}{4\pi|x-y|^3}\bigg) \, dy\\
- \PV \int_\Om F_l(y)\, \frac{|x-y|^2\de_{jm}-3(x_j-y_j)(x_m-y_m)}{4\pi|x-y|^5} \, dy
\end{multline*}
can be estimated using the fact that $\eta_\de(r)=1$ for $r>\de$ and the mean value theorem as
\begin{align}
|M|&\leq \int_{B_\de(x)}|F_l(y)-F_l(x)|\,\bigg|\pd_{x_j}\bigg((1-\eta_\de(|x-y|))\,
\frac{x_m-y_m}{4\pi|x-y|^3}\bigg)\bigg| \, dy\notag\\
&\leq C \int_{B_\de(x)}\frac{|F_l(y)-F_l(x)|}{|x-y|^3} \, dy\notag\\
&\leq C\|\nabla F_l\|_{L^\infty(\Om)}\int_{B_\de(x)}\frac1{|x-y|^2}\,
dy\notag\\
&\leq C\de\,.\label{M}
\end{align}

Since $A^{j,\de}(x)$ obviously converges to the field
\[
A^j(x):=\int_{\pd\Om} \nu_j(y)\, \frac{x-y}{4\pi|x-y|^3} \, d\si(y)
\]
for all $x\in\Om$, Equations~\eqref{Dw} and~\eqref{M} show that the
derivative $\pd_j w_k$ is indeed given by the formula~\eqref{Dwbien} of the
statement. As $A^j$ is obviously smooth in~$\Om$, it only remains to
show that $A^j(x)$ is bounded when $x\to\pd\Om$. 

For this, let us take an arbitrary point of the boundary, which we can
assume to be the origin. Rotating the coordinate axes if necessary,
one can parametrize $\pd\Om$ in a neighborhood $U_\rho$ of the origin
as the graph
\begin{equation}\label{graph}
Y\in D_\rho \mapsto (Y,h(Y))\in \RR^3\,,
\end{equation}
where $D_\rho:=\{Y\in\RR^2:|Y|<\rho\}$ is the two-dimensional disk of a small radius~$\rho$
and the function $h$ satisfies
\[
h(0)=0\,,\qquad \nabla h(0)=0\,.
\]
Let us now analyze the behavior of $A^j(x)$ when $x=(0,0,-t)$ and $t\to
0^+$. Since the unit normal and the surface measure can be written in
terms of~$h$ as
\[
\nu=\frac{(-\nabla h,1)}{\sqrt{1+|\nabla h|^2}}\,,\qquad d\si=
\sqrt{1+|\nabla h|^2}\, dY
\]
and $x\in U_\rho$ for small enough~$t$, it follows that
\begin{align}
|A_j(x)|&\leq\bigg| \int_{\pd\Om\cap U_\rho} \nu_j(y)\, \frac{x-y}{4\pi|x-y|^3}
\, d\si(y)\bigg| +\bigg|\int_{\pd\Om\minus U_\rho} \nu_j(y)\,
\frac{x-y}{4\pi|x-y|^3} \, d\si(y)\bigg|\notag\\
&\leq\frac1{4\pi}\bigg|\int_{D_\rho}\bar\nu_j(Y)\, \frac{(Y,t+ h(Y))}{(|Y|^2+(t+ h(Y))^2)^{3/2}} \,
dY \bigg|+ C\,,\label{Aj}
\end{align}
where we have used that the second integral is bounded by a constant
independent of~$t$ and
\[
\bar\nu_j:= \sqrt{1+|\nabla h|^2}\, \nu_j
\]

It is clear then that
\begin{align*}
\bigg|\int_{D_\rho}\bar\nu_j(Y)\,& \frac{Y}{(|Y|^2+(t+ h(Y))^2)^{3/2}} \,
dY \bigg|=
\bigg|\int_{D_\rho}\frac{\bar\nu_j(0) \, Y+O(|Y|^2)+t\, O(Y)}{(|Y|^2+t^2)^{3/2}}\,
dY \bigg|\\
&\qquad \qquad \qquad \qquad \qquad\leq\bigg| \bar\nu_j(0)\int_{D_\rho} \frac{ Y}{(|Y|^2+t^2)^{3/2}}\,
dY\bigg| + C\int_{D_\rho}\frac{dY}{|Y|}\\
&\qquad \qquad \qquad \qquad \qquad\leq C\rho\,,
\end{align*}
where we have used that the first integral in the second line vanishes
by parity. Likewise,
\begin{align*}
\bigg|\int_{D_\rho}\bar\nu_j(Y)\,& \frac{t+ h(Y)}{(|Y|^2+(t+ h(Y))^2)^{3/2}} \,
dY \bigg|=
\bigg|\int_{D_\rho}\frac{\bar\nu_j(0) \, t+O(|Y|^2)+t\, O(Y)}{(|Y|^2+t^2)^{3/2}}\,
dY \bigg|\\
&\qquad \qquad \qquad \qquad \qquad\leq  2\pi
  t|\bar\nu_j(0)|\int_0^\rho \frac{ r\, dr}{(r^2+t^2)^{3/2}} + C\int_{D_\rho}\frac{dY}{|Y|}\\
&\qquad \qquad \qquad \qquad \qquad\leq 2\pi 
|  \bar\nu_j(0)|\int_0^\infty \frac{ r\, dr}{(r^2+1)^{3/2}} + C\rho\\
&\qquad \qquad \qquad \qquad \qquad\leq C\,,
\end{align*}
Hence we infer from~\eqref{Aj} that $\|A^j\|_{L^\infty(\Om)}\leq C$
and the lemma follows. 
\end{proof}

Lemma~\ref{L.Dw} immediately yields the following characterization of
the divergence and curl of~$w$:

\begin{proposition}\label{P.curl}
The vector field $w$ is divergence-free in $\RR^3$ and its curl at a point $x\in\Om$ is given by
\begin{align}\label{curlw}
\curl w(x)= F(x)+ \nabla \int_\Om \frac{\Div F(y)}{4\pi|x-y|}\, dy -
\nabla\int_{\pd\Om}
 \frac{F(y)\cdot \nu(y)}{4\pi|x-y|}\, d\si(y)\,.
\end{align}
\end{proposition}

\begin{proof}
It follows from the proof of Lemma~\ref{L.Dw} that $\Div
w=\lim_{\de\to 0} \Div w^\de$, with $w^\de$ defined as before. For
simplicity, let us write 
\[
E:=\frac1{4\pi|x-y|}\,.
\]
Since
Equation~\eqref{Dw} shows that
\begin{align*}
\Div w^\de&= \ep_{klm}\int_\Om F_l(y)\, \pd_{x_k}\bigg(\eta_\de(|x-y|)\,
\frac{x_m-y_m}{4\pi|x-y|^3}\bigg) \, dy\\
&= \ep_{klm}\int_\Om F_l(y)\, \bigg(\eta_\de'(|x-y|)\frac{(x_k-y_k)(x_m-y_m)}{4\pi|x-y|^4}-\eta_\de(|x-y|)\,
\pd_{x_k}\pd_{x_m}E\bigg) \, dy\\
&=0\,,
\end{align*}
it follows that $\Div w=0$ everywhere.

To compute $\curl w$, let us begin by computing the $i\th$ component
of $\curl w^\de$ using again Equation~\eqref{Dw} and the fact that
$\ep_{ijk}\ep_{klm}= \de_{il}\de_{jm}-\de_{im}\de_{jl}$:
\begin{align*}
\ep_{ijk}\pd_j w^\de_k(x)& =\int_\Om F_i(y)\, \pd_{x_j}\bigg(\eta_\de(|x-y|)\,
\frac{x_j-y_j}{4\pi|x-y|^3}\bigg) \, dy \\
&\qquad \qquad \qquad \qquad \qquad+ \int_\Om F_j(y)\, \pd_{x_j}\big(\eta_\de(|x-y|)\,
\pd_{x_i}E\big) \, dy\\
&=\pd_{x_j}\int_\Om F_i(y)\, \eta_\de(|x-y|)\,
\frac{x_j-y_j}{4\pi|x-y|^3} \, dy \\
&\qquad \qquad \qquad \qquad \qquad- \int_\Om F_j(y)\, \pd_{y_j}\big(\eta_\de(|x-y|)\,
\pd_{x_i}E\big) \, dy\\
&=\pd_{x_j}\int_\Om F_i(y)\, \eta_\de(|x-y|)\,
\frac{x_j-y_j}{4\pi|x-y|^3} \, dy \\
&- \int_{\pd\Om} F\cdot\nu (y)\,\eta_\de(|x-y|)\,
\pd_{x_i}E \, d\si(y) + \int_{\Om} \Div F(y)\, \eta_\de(|x-y|)\,
\pd_{x_i}E \, dy\,.
\end{align*}
Taking the limit $\de\to0$ and using that $E$ is a fundamental
solution of the Laplacian one then obtains
\begin{align*}
(\curl w)_i(x)&= \lim_{\de\to0} \ep_{ijk}\pd_j w_k^\de(x)\\
&=\pd_{x_j}\int_\Om F_i(y)\, 
\frac{x_j-y_j}{4\pi|x-y|^3} \, dy - \int_{\pd\Om} F\cdot\nu (y)\,
\pd_{x_i}E \, d\si(y) \\
&\qquad\qquad \qquad\qquad \qquad\qquad \qquad \qquad \qquad 
+ \int_{\Om} \Div F(y)\, 
\pd_{x_i}E \, dy\\
&=F_i(x)-\pd_{x_i}\int_{\pd\Om}F\cdot \nu\, E\, d\si(y)+\pd_{x_i}\int_\Om
  \Div F\, E\, dy\,.
\end{align*}
The proposition then follows.
\end{proof}

We are now ready to state the basic $L^p$ estimate for~$w$:

\begin{proposition}\label{P.Dw}
For any nonnegative integer~$n$ and $1<p<\infty$, the field~$w$ can be estimated as
\[
\|w\|_{W^{n+1,p}(\Om)}\leq C\|F\|_{W^{n,p}(\Om)}\,.
\]
\end{proposition}
\begin{proof}
Notice that the principal value integral appearing in the first term of the formula~\eqref{Dwbien} is the action
of a Calder\'on--Zygmund operator in $\RR^3$ on $F_l\, 1_\Om$, where $1_\Om$
denotes the indicator function of the domain. As the field $A^j$ is
bounded, it is then standard that
\begin{equation}\label{estW1}
\|w\|_{W^{1,p}(\Om)}\leq C\|F\|_{L^p(\Om)}\,.
\end{equation}
This is the bound of the statement with $n=0$.

We claim that, if $Z^1,\dots, Z^n$ are smooth vector fields in $\BOm$
which are tangent to $\pd\Om$, then
\begin{equation}\label{tang.bound}
\|Z^1\cdots Z^n w\|_{W^{1-\frac1p,p}(\pd\Om)} \leq  C\|F\|_{W^{n,p}(\Om)}\,,
\end{equation}
where, as it is customary, we are regarding the vector field $Z^j$ as a first-order
differential operator (namely, $Z^j=Z^j_i(x)\, \pd_{x_i}$) and the
function~$F$ can be taken smooth.
It is easy to see that the proposition readily follows from this
estimate. Indeed, 
Proposition~\ref{P.curl} ensures that $w$ is
divergence-free, so we can take the curl of~\eqref{curlw} to find that
\[
\De w=-\curl F
\]
in~$\Om$. Since~\eqref{tang.bound} means that
$\|w\|_{W^{n+1-\frac1p,p}(\pd\Om)}\leq C\|F\|_{W^{n,p}(\Om)}$, standard elliptic estimates then yield
\begin{align*}
\|w\|_{W^{n+1,p}(\Om)}&\leq C\big(\|\curl F\|_{W^{n-1,p}(\Om)} +
\|w\|_{W^{1,p}(\Om)}+ \|w\|_{W^{n+1-\frac1p,p}(\pd\Om)}\big)\\
&\leq C\|F\|_{W^{n,p}(\Om)}\,,
\end{align*}
as claimed. 

Hence it only remains to prove~\eqref{tang.bound}. In view of the
trace inequality
\[
\|f\|_{W^{s,q}(\pd\Om)}\leq C\|f\|_{W^{s+\frac1q,q}(\Om)}\,,
\]
it suffices to show that for any~$j$ and~$n$ one can write
\[
\pd_j(Z^1\cdots Z^n w_k)=\cI_n+\cJ_n\,,
\]
where the terms $\cI_n$ and $\cJ_n$ (which depend on $j$, $k$ and~$n$) are respectively bounded as
\[
\|\cI_n\|_{W^{s,q}(\pd\Om)}\leq C\|F\|_{W^{s+n-1,q}(\pd\Om)}\,,\qquad
\|\cJ_n\|_{L^p(\Om)}\leq C\|F\|_{W^{n,p}(\Om)}
\]
for all real $s$ and all $1<q<\infty$. In fact, it is slightly more convenient to prove an analogous estimate for the quantity
\[
Z^1\cdots Z^n \pd_j w_k =\cI_1+\cI_2\,;
\]
this clearly suffices for our purposes as the commutator term
\[
\pd_j(Z^1\cdots Z^n w_k)-Z^1\cdots Z^n \pd_j w_k
\]
only involves $n$th order derivatives of~$w$, of which at least $n-1$
are taken along tangent directions on the boundary.  The ideas of the
proof are mostly standard, but we will provide a sketch of the proof
as we have not found a suitable reference in the literature.

Let us start with the case
$n=1$. We can differentiate the
formula~\eqref{Dw} to obtain, for any tangent vector field~$Z$,
\begin{align*}
Z\pd_j w^\de_k(x)& =\ep_{klm}Z_i(x)\int_\Om F_l(y)\, \pd_{x_i}\pd_{x_j}\bigg(\eta_\de(|x-y|)\,
\frac{x_m-y_m}{4\pi|x-y|^3}\bigg) \, dy\\
&=-\ep_{klm}Z_i(x)\int_\Om F_l(y)\, \pd_{y_i}\pd_{x_j}\bigg(\eta_\de(|x-y|)\,
\frac{x_m-y_m}{4\pi|x-y|^3}\bigg) \, dy\\
&=:\cI_1+\cJ_1\,,
\end{align*}
where
\begin{align*}
\cI_1&:= -\ep_{klm}\int_{\pd\Om} Z(x)\cdot \nu(y)F_l(y)\, \pd_{x_j}\bigg(\eta_\de(|x-y|)\,
\frac{x_m-y_m}{4\pi|x-y|^3}\bigg) \, d\si(y)\,,\\
\cJ_1&:= \ep_{klm}Z_i(x)\int_\Om \pd_iF_l(y)\, \pd_{x_j}\bigg(\eta_\de(|x-y|)\,
\frac{x_m-y_m}{4\pi|x-y|^3}\bigg) \, dy\,.
\end{align*}
Writing the volume integral as
\begin{multline*}
\cJ_1=\ep_{klm}Z_i(x)\int_\Om [\pd_iF_l(y)-\pd_i F_l(x)]\, \pd_{x_j}\bigg(\eta_\de(|x-y|)\,
\frac{x_m-y_m}{4\pi|x-y|^3}\bigg) \, dy\\
- ZF_l(x)\, \ep_{klm}A^{j,\de}_m(x)
\end{multline*}
to obtain a principal-value-type formula, it is clear in view of
the boundedness of~$A^j_m$ that the $L^p$ norm of~$\cJ_1$ is controlled
in terms of~$F$ with a $\de$-independent constant:
\[
\|\cJ_1\|_{L^p(\Om)}\leq C\|F\|_{W^{1,p}(\Om)}\,.
\]

To analyze the boundary term, $\cI_1$, we will next restrict our
attention to points~$x$ lying on~$\pd\Om$. We can perform the analysis
locally, parametrizing a portion of the boundary as a graph using the
notation~\eqref{graph}, which amounts to writing $x=(X,h(X))$, possibly after
a rotation of the coordinate axes. A basis for the space of tangent vectors at
the point~$x$ is
\[
T_1(X):=(1,0,\pd_1 h(X))\,,\qquad T_2(X):=(0,1,\pd_2h(X))\,.
\]
Abusing the notation to denote by $f(X)$ the value of a function $f(x)$ at the
point $x=(X,h(X))\in\pd\Om$, one can therefore write
\[
Z(X)=a_1(X)\, T_1(X)+a_2(X)\, T_2(X)
\]
in terms of two smooth functions $a_j(X)$. This implies that
\[
Z(X)\cdot \nu(Y)=\frac{a_1(X)\,[\pd_1 h(X)-\pd_1h(Y)]+ a_2(X)\,[\pd_2
  h(X)-\pd_2h(Y)]}{\sqrt{1+|\nabla h(Y)|^2}}\,.
\] 
Using the cancellation that stems from this formula it is not hard to see that, as the singular part of
$\cI_1$ is
\begin{multline*}
\int_{D_\rho} g(Y)\, F_l(Y)\, \bigg[Z(X)\cdot
\nu(Y)\,\eta_\de(|X-Y|)\,\frac{|X-Y|^2\de_{jm}-3(X_j-Y_j)(X_m-Y_m)}{|X-Y|^5}\\
+
\cO\Big(\frac1{|X-Y|}\Big)\bigg]\, dY
\end{multline*}
with $g(Y)$ a smooth function, $\cI_1$ defines a singular integral
operator on the boundary, so we have
\[
\|\cI_1\|_{L^q(\pd\Om)}\leq C\|F\|_{L^q(\pd\Om)}
\]
for all $1<q<\infty$. Furthermore, the coefficients are smooth and the
derivatives have the right singularities, so a straightforward
computation shows that $\cI_1$ behaves like a zeroth order
pseudodifferential operator on the boundary, leading to the estimate
\[
\|\cI_1\|_{W^{s,q}(\pd\Om)}\leq C\|F\|_{W^{s,q}(\pd\Om)}\,.
\]
The case $n=1$ then follows.

In the case of~$n\geq2$, the proof goes along the same lines but there is
another kind of term that one needs to consider. To see this, we
next consider the case $n=2$, which illustrates all the difficulties
that appear in the general case. For this we take another vector field $Z'$ that is
tangent to the boundary and differentiating the formula for $Z\pd_j
w_k^\de$ to get
\begin{multline*}
Z'Z\pd_j w_k^\de(x)= \ep_{klm}Z'_n(x) \pd_{x_n}\int_\Om Z_i(x) \,\pd_iF_l(y)\, \pd_{x_j}\bigg(\eta_\de(|x-y|)\,
\frac{x_m-y_m}{4\pi|x-y|^3}\bigg) \, dy\\
- \ep_{klm}Z'_n(x)\pd_{x_n}\int_{\pd\Om} Z(x)\cdot \nu(y)F_l(y)\, \pd_{x_j}\bigg(\eta_\de(|x-y|)\,
\frac{x_m-y_m}{4\pi|x-y|^3}\bigg) \, d\si(y)\,.
\end{multline*}
The volume integral can be dealt with just as in the case of $n=1$. Indeed,
the action of $\pd_{x_n}$ on $Z_i(x)$ is harmless, while when $\pd_{x_n}$
acts on the singular term,  
one just replaces the derivative $\pd_{x_n}$ by $-\pd_{y_n}$ and integrates
by parts. This yields another volume integral that can be related to a
principal value as above, leading to $W^{2,p}(\Om)\to L^p(\Om)$ bounds, and a boundary integral with the same
structure as above (with $DF$ playing the role of $F$), which leads to
$W^{s+1,p}(\pd\Om)\to W^{s,q}(\pd\Om)$ bounds.

The estimates for the surface integral are also as above when the
derivative $\pd_{x_n}$ acts on the field $Z(x)$. When it acts on the
singular term, however, one needs to refine the argument a little
bit. Again we start by replacing the $\pd_{x_n}$ by $\pd_{y_n}$, so we
have to control the integral
\[
\cI:=\int_{\pd\Om} Z(x)\cdot \nu(y)F_l(y)\, \pd_{y_n}\pd_{x_j}\bigg(\eta_\de(|x-y|)\,
\frac{x_m-y_m}{4\pi|x-y|^3}\bigg) \, d\si(y)\,.
\]
The point is that, as one can decompose the $n$th unit vector as
\[
e_n= T(y)+b(y)\, \nu(y)\,,
\]
where $T$ is a tangent vector and $\nu$ is the unit normal, one
can also write
\[
\pd_{y_n}=T(y)+ b(y)\,\nu(y)\cdot \nabla_y\,,
\]
where the vector field $T$ is here interpreted as a differential
operator as before. 

We can now integrate the tangent
field by parts to arrive at
\begin{multline*}
\cI=\int_{\pd\Om} [Z(x)\cdot \nu(y)]TF_l(y)\, \pd_{y_n}\pd_{x_j}\bigg(\eta_\de(|x-y|)\,
\frac{x_m-y_m}{4\pi|x-y|^3}\bigg) \, d\si(y)\\
+\int_{\pd\Om} \cO(1)\, [Z(x)\cdot \nu(y)]TF_l(y)\, \pd_{y_n}\pd_{x_j}\bigg(\eta_\de(|x-y|)\,
\frac{x_m-y_m}{4\pi|x-y|^3}\bigg) \, d\si(y) \\
+\int_{\pd\Om} [Z(x)\cdot \nu(y)]TF_l(y)\,  b(y)\,[\nu(y)\cdot \nabla_y]\pd_{x_j}\bigg(\eta_\de(|x-y|)\,
\frac{x_m-y_m}{4\pi|x-y|^3}\bigg) \, d\si(y)\,.
\end{multline*}
The first term admits bounds $W^{s+1,q}(\pd\Om)\to W^{s,q}(\pd\Om)$
just as before, the second term (which we get from tangential
derivatives that act on the unit normal and from the divergence of~$T$
as a vector field on~$\pd\Om$) is clearly bounded $W^{s,q}(\pd\Om)\to W^{s,q}(\pd\Om)$, and
we just have to control the last integral. For this we need another
cancellation, which hinges on the well-known fact that
\[
\nu(y)\cdot (x-y)
\]
is of order $|x-y|^2$ when $x,y\in \pd\Om$. This appears here because
the singular term in the last integral is
\begin{multline*}
[Z(x)\cdot \nu(y)][\nu(y)\cdot \nabla_y]\pd_{x_j}\bigg(\eta_\de(|x-y|)\,
\frac{x_m-y_m}{|x-y|^3}\bigg)=\\
 \eta_\de(|x-y|)\,[Z(x)\cdot \nu(y)]\,[\nu(y)\cdot(x-y)]
 \frac{-9|x-y|^2\de_{jm}+15(x_j-y_j)(x_m-y_m)}{|x-y|^7}\\
+\cO\Big(\frac1{|x-y|}\Big)\,.
\end{multline*}
This readily yields the estimate $W^{s+1,q}(\pd\Om)\to
W^{s,q}(\pd\Om)$. The general case is handled by repeatedly applying these ideas.
\end{proof}

\section{The extension operator $\cE_T$}
\label{S.extension}

In this section we will construct an extension operator that will
be of use in the construction of the kernel $K_\Om(x,y)$. For this,
let us denote by $\rho:\RR^3\to\RR$
the signed distance to the set~$\pd\Om$, which is smooth in the
set $\widetilde\cU:=\rho^{-1}((-\rho_0,\rho_0))$ provided that $\rho_0$ is small
enough. Notice that $\widetilde\cU$ is a tubular neighborhood of the
boundary~$\pd\Om$, which one can then identify $\widetilde\cU$ with
$\pd\Om\times(-\rho_0,\rho_0) $ via a diffeomorphism
\[
x\in \widetilde\cU\mapsto (x',\rho)\in \pd\Om\times(-\rho_0,\rho_0)\,.
\]
We will often write $\rho_x\equiv\rho(x)$. 

Taking local normal
coordinates $X\equiv (X_1,X_2)$ on~$\pd\Om$, the Euclidean metric reads
as
\begin{equation}\label{metric}
ds^2|_{\widetilde\cU}= G_\rho+d\rho^2\,,
\end{equation}
where
\[
G_\rho :=h_{ij}(X,\rho)\, dX_i\, dX_j
\]
defines a $\rho$-dependent metric on~$\pd\Om$ that coincides with the induced
surface metric on~$\pd\Om$ at $\rho=0$. Hence the
volume reads as
\[
dx= d\si_\rho(x')\, d\rho\,,
\]
where $d\si_\rho$ is a $\rho$-dependent metric on~$\pd\Om$ that can be
written in local coordinates as
\[
d\si_\rho=\sqrt{\det(h_{ij}(X,\rho))}\, dX_1\, dX_2\,.
\]
Obviously the connection with the surface measure on~$\pd\Om$ is
\[
d\si_\rho=(1+O(\rho))\, d\si\,.
\]

Consider the portion of the tubular neighborhood $\widetilde\cU$
contained in~$\Om$,
\[
\cU:=\widetilde\cU\cap\Om=\rho^{-1}((0,\rho_0))\,.
\]
Let us denote by $\om^\perp$ the $\rho$-component of a vector field~$\om$,
so that one can decompose~$\om$
in~$\cU$ as
\[
\om= \om^\parallel\,+\om^\perp\, \pd_\rho,
\]
where $\om^\parallel$ is orthogonal to $\pd_\rho$.
To put it differently,
\[
\om^\perp:= \om\cdot\nabla\rho\,,\qquad \om^\parallel:=\om-\om^\perp\, \nabla\rho\,.
\]

Consider the operator $T_0$ defined, for $x\in\pd\Om$, by the
integral
\begin{equation}\label{defT0}
T_0f(x):=\int_{\pd\Om}\frac{(x-y)\cdot \nu(x)}{4\pi|x-y|^3}\, f(y)\, d\si(y)\,.
\end{equation}
It is well known (see e.g.~\cite[Section 7.11]{Taylor}) that, under
the assumption that the boundary is smooth, $T_0$ is a
pseudodifferential operator on $\pd\Om$ of order $-1$. (More
generally, notice that for a domain with $C^2$~boundary, the kernel of
the operator is bounded by $C/|x-y|$.) In this paper we will employ operators on~$\pd\Om$ of the form
\begin{equation}\label{defT}
T f(x)=\int_{\pd\Om}K_T(x,y)\, f(y)\, d\si(y)
\end{equation}
with 
\begin{equation}\label{kernelKT}
K_T(x,y):= \frac{(x-y)\cdot \nu(x)}{4\pi|x-y|^3}+ K_T'(x,y)
\end{equation}
and $K_T'$ bounded on~$\pd\Om\times\pd\Om$. We will also assume that
the derivative of $K_T'(x,y)$ is bounded by $C/|x-y|$.

If $\chi(t)$ is a smooth cut-off function that is equal to~1 for
$t<\rho_0/2$ and which vanishes for $t>\rho_0$, one can exploit the
above identification of~$\cU$ with $\pd\Om\times(0,\rho_0)$ to
define a vector field $\cE_T\om:\Om\to\RR^3$ as
\begin{equation}\label{tom}
\cE_T\om(x):=  \bigg[\chi(\rho_x)
\int_{\pd\Om}\bigg(\frac{(x'-y)\cdot \nu(x')}{4\pi|x-y|^3}+
K_T'(x',y)\bigg)\,\om^\perp(y,\rho_x)\, d\si_{\rho_x}(y)\bigg]\, \pd_\rho\,,
\end{equation}
where $K_T'$ is the kernel defined in~\eqref{kernelKT}. Notice
that $\cE_T\om$ is supported in~$\cU$. 

The basic properties of the extension operator that we will need later
are the following:

\begin{proposition}\label{P.PDO}
Let $T$ be an operator of the form~\eqref{defT}. For any
divergence-free vector field $\om\in C^1(\BOm)$ one has
\[
T(\om\cdot\nu)=(\cE_T\om)\cdot\nu\,,
\]
where~$\cE_T\om$ is defined by~\eqref{tom}, and for each $x\in\Om$ the divergence of~$\cE_T\om$
can be written as 
\begin{align*}
\Div \cE_T\om(x)&=\int_{\pd\Om} K_{T,\mathrm{div}}(x,y)\,\om(y,\rho_x)\, d\si_{\rho_x}(y)
\end{align*}
with a kernel of the form
\begin{multline}\label{Kdiv}
K_{T,\mathrm{div}}(x,y)\,\om(y,\rho_x):=\chi(\rho_x) \bigg[\frac
{3(x'-y)\cdot\nu(x')\,
  (y-x)\cdot\nu(x')}{4\pi|x-y|^5}\,\om^\perp(y,\rho_x) \\+ \frac
{3(x'-y)\cdot\nu(x')\,
  (x-y)-|x-y|^2\nu(x')}{4\pi|x-y|^5}\cdot
\om^\parallel(y,\rho_x)\bigg]+ \tK_T(x,y)\, \om(y,\rho_x)\,.
\end{multline}
where $|\tK_T(x,y)|\leq C/|x-y|$ and is supported in
$\overline\cU\times\pd\Om$.
\end{proposition}

\begin{proof}
Since $x\not\in\pd\Om$, one can easily compute the divergence
of~$\tom:=\cE_T\om$ as
\begin{align}
\Div {\tom}(x) &= \pd_{\rho_x}\bigg[\chi(\rho_x)
\int_{\pd\Om}\bigg(\frac{(x'-y)\cdot
  \nu(x')}{4\pi|x-y|^3}+ K_T'(x',y)\bigg)\,\om^\perp(y,\rho_x)\,
d\si_{\rho_x}(y)\bigg]\notag\\
&=\chi(\rho_x) \int_{\pd\Om}\pd_{\rho_x}\bigg(\frac{(x'-y)\cdot
  \nu(x')}{4\pi|x-y|^3}\,\om^\perp(y,\rho_x)\bigg)\,
d\si_{\rho_x}(y)+J_2\,,\label{intJ1}
\end{align}
where $J_2$ denotes a term of the form
\[
J_2=\int_{\pd\Om}\tK_T(x,y)\, \om(y,\rho_x)\, d\si(y)
\]
with a kernel $\tK_T$ as in the statement. Let us denote by $J_1$ the
integral that appears in the second line of~\eqref{intJ1}. To simplify
the expression of $J_1$, we shall use that, by~\eqref{metric}, one can write the
divergence of~$\om$ as
\[
0=\Div \om(y,\rho_x)=\pd_{\rho_x}\om^\perp(y,\rho_x)+\nabla^\parallel\cdot\om^\parallel(y,\rho_x)\,,
\]
where $\nabla^\parallel\cdot\om^\parallel(y,\rho_x)$ is the divergence of the
field $\om^\parallel$ (understood as a tangent vector field
on~$\pd\Om$) with respect to the divergence operator on~$\pd\Om$
associated with the measure $d\si_{\rho_x}$. This allows us to write
\begin{align*}
J_1&:= \int_{\pd\Om}\pd_{\rho_x}\bigg(\frac{(x'-y)\cdot
  \nu(x')}{4\pi|x-y|^3}\,\om^\perp(y,\rho_x)\bigg)\,
d\si_{\rho_x}(y)\\
&= \int_{\pd\Om}\pd_{\rho_x}\bigg(\frac{(x'-y)\cdot
  \nu(x')}{4\pi|x-y|^3}\bigg)\,\om^\perp(y,\rho_x)\,
d\si_{\rho_x}(y) \\
&\qquad \qquad \qquad \qquad\qquad+ \int_{\pd\Om}\frac{(x'-y)\cdot
  \nu(x')}{4\pi|x-y|^3}\,\pd_{\rho_x}\om^\perp(y,\rho_x)\,
d\si_{\rho_x}(y)\\
&=\int_{\pd\Om}\pd_{\rho_x}\bigg(\frac{(x'-y)\cdot
  \nu(x')}{4\pi|x-y|^3}\bigg)\,\om^\perp(y,\rho_x)\,
d\si_{\rho_x}(y) \\
&\qquad \qquad \qquad \qquad\qquad- \int_{\pd\Om}\frac{(x'-y)\cdot
  \nu(x')}{4\pi|x-y|^3}\,\nabla^\parallel\cdot\om^\parallel(y,\rho_x)\,
d\si_{\rho_x}(y)\\
&=\int_{\pd\Om}(\om^\perp(y,\rho_x)\,\pd_{\rho_x}+
\om^\parallel(y,\rho_x)\cdot\nabla)\bigg(\frac{(x'-y)\cdot
  \nu(x')}{4\pi|x-y|^3}\bigg)\,
d\si_{\rho_x}(y) \\
&= \int_{\pd\Om}\bigg[\frac
{3(x'-y)\cdot\nu(x')\,
  (y-x)\cdot\nu(x')}{4\pi|x-y|^5}\,\om^\perp(y,\rho_x) \\
&\qquad\qquad + \frac
{3(x'-y)\cdot\nu(x')\,
  (x-y)-|x-y|^2\nu(x')}{4\pi|x-y|^5}\cdot
\om^\parallel(y,\rho_x)\bigg]\,
d\si_{\rho_x}(y) \,,
\end{align*}
where we have integrated by parts the tangential divergence
$\nabla^\parallel\cdot \om^\parallel$. The proposition then follows.
\end{proof}

\section{Construction of the solution}
\label{S.construction}

Given the divergence-free field~$\om$ in $W^{k,p}(\Om)$, our objective in this section is to construct a vector field~$u$ that
satisfies the equations
\[
\curl u=\om\,,\qquad \Div u=0
\]
in~$\Om$ and is tangent to the boundary. To this end, let us start by
considering the vector field~$v$ on~$\Om$ defined as
\begin{align}
v(x)&:=v_1(x)+v_2(x)\,,\notag\\[1mm]
v_1(x)&:=\int_\Om \frac{\om(y)
  \times(x-y)}{4\pi|x-y|^3}\, dy\,,\label{defv1}\\[1mm]
v_2(x)&:= -\int_\Om \frac{\nabla \cS f(y) \times(x-y)}{4\pi|x-y|^3}\, dy\,,\label{defv2}
\end{align}
where for $x\in\Om$ we define the scalar function~$\cS f$ through the single layer potential
\[
\cS f(x):=-\int_{\pd\Om}\frac{f(y)}{4\pi|x-y|}\, d\si(y)\,,
\]
with $f$ a function on~$\pd\Om$ to be determined. 

Since $\De(\cS
f)=0$ in $\Om$ for any function~$f$, it follows from
Proposition~\ref{P.curl} that $v$~is divergence-free and
that its curl is given by
\begin{align}
\curl v(x)&= \om(x)-\nabla \cS f(x)+\nabla \int_{\pd\Om}\frac{\pd_\nu
            \cS f(y)-\om\cdot \nu(y)}{4\pi|x-y|}\, d\si(y)\notag\\
&=\om(x)+\nabla\int_{\pd\Om}\frac{f(y)+\pd_\nu
            \cS f(y)-\om\cdot \nu(y)}{4\pi|x-y|}\, d\si(y)\,.\label{curlv}
\end{align}
As a side remark, observe that since we have only defined $\cS f$ on~$\Om$, the fact that the
derivative of the extension of $\cS f$ to the whole space~$\RR^3$ is
discontinuous across~$\pd\Om$ does not play a role here. Consequently, $\pd_\nu \cS f(x)$
will necessarily stand for the interior derivative of~$\cS f$ at the boundary point~$x$ in the direction of the
outer normal, that is, 
\[
\pd_\nu \cS f(x):=\lim_{y\in \Om, \; y\to x} \nu(x)\cdot \nabla \cS f(y)\,,
\]
where for all $y\in\Om$ one has
\[
\nabla \cS f(y) =\int_{\pd\Om}f(z)\frac{y-z}{4\pi|y-z|^3}\, d\si(z)\,.
\]

Consider the operator $T_0$ defined in~\eqref{defT0}, which appears in the
analysis of the normal derivative of~$\cS f$ at the boundary through the
formula
\[
\pd_\nu \cS f = \big(T_0-\tfrac12I \big)f\,.
\]
In view of Equation~\eqref{curlv}, our goal now is to choose the function~$f$ so that
\[
f+\pd_\nu \cS f-\om\cdot\nu =0\,,
\]
or equivalently
\begin{equation}\label{eqf}
\big(\tfrac12I +T_0\big) f=\om\cdot\nu\,.
\end{equation}

Since $\frac12I+T_0$ is precisely the operator that one needs to invert
in order to solve the exterior Neumann boundary value problem for the
Laplacian, it is well known (see e.g.~\cite[Section 3.E]{Folland})
that there is a function $f$ satisfying~\eqref{eqf} if and only if
\begin{equation}\label{cond}
\int_{\pd\Om_j}\om\cdot\nu\, d\si=0\,,\qquad 1\leq j\leq m-1\,.
\end{equation}
Here $\Om_1,\dots,\Om_{m-1}$ are the bounded connected components of
$\RR^3\backslash\BOm$. Since each $\pd\Om_j$ is a connected component
of~$\pd\Om$, the hypothesis~\eqref{hypothesis} ensures
that the condition~\eqref{cond} holds, so one can find a function~$f$ on~$\pd\Om$ 
satisfying~\eqref{eqf}. Notice, moreover, that~\eqref{eqf} implies
that~$f$ can be written as
\begin{equation}\label{formulaf}
f=(2-4T)(\om\cdot\nu)\,,
\end{equation}
with $T$ an operator of the form~\eqref{defT}.

Since $T$ is a pseudodifferential operator on~$\pd\Om$ of order
$-1$, Equation~\eqref{formulaf} yields the estimate 
\[
\|f\|_{W^{k-\frac1p,p}(\pd\Om)}\leq C \|\om\cdot
\nu\|_{W^{k-\frac1p,p}(\pd\Om)}\leq C \|\om\|_{W^{k,p}(\Om)}\,,
\]
and therefore, by the properties of the single layer potential,
\[
\|\cS f\|_{W^{k+1,p}(\Om)}\leq C \|f\|_{W^{k-\frac1p,p}(\pd\Om)} \leq C \|\om\|_{W^{k,p}(\Om)}\,.
\]

With this function~$f$, by construction one then has that
\[
\curl v=\om\,,\qquad \Div v=0\,,
\]
while Proposition~\ref{P.Dw} ensures that~$v$ is bounded as
\[
\|v\|_{W^{k+1,p}(\Om)}\leq C\|\om\|_{W^{k,p}(\Om)}+ C\|\nabla\cS
f\|_{W^{k,p}(\Om)}\leq C\|\om\|_{W^{k,p}(\Om)}\,.
\]
Hence one can now find a solution to the problem~\eqref{problem} by
setting
\begin{equation}\label{defu}
u:=v-\nabla \vp\,,
\end{equation}
where $\vp$ is a solution to the Neumann boundary value problem
\[
\De\vp=0\quad\text{in }\Om\,,\qquad \pd_\nu \vp=v\cdot \nu\,.
\]
It is well-known that the solution exists and is unique up to an additive constant because
\[
\int_{\pd\Om} v\cdot \nu\, d\si=\int_\Om\Div v\, dx=0\,.
\]
Moreover, the function $\vp$ can be written as a single layer potential
\[
\vp=\cS g\,,
\]
where the function~$g$ and $v\cdot \nu$ are 
related through the operator~$T_0$ as
\[
\big(\tfrac12I -T_0\big)g=-v\cdot \nu\,.
\]
Therefore,
\begin{equation}\label{formulag}
g=-(2+4\widetilde T)(v\cdot \nu)
\end{equation}
with $\widetilde T$ a pseudodifferential operator on~$\pd\Om$ of the form~\eqref{defT}, so 
\[
\|g\|_{W^{k+1-\frac1p,p}(\pd\Om)}\leq C \|v\cdot
\nu\|_{W^{k+1-\frac1p,p}(\pd\Om)}\leq C \|v\|_{W^{k+1,p}(\Om)}\leq C \|\om\|_{W^{k,p}(\Om)}
\]
By the properties of single layer potentials it then follows that
\begin{equation*}
\|\nabla\vp\|_{_{W^{k+1,p}(\Om)}}\leq \|\cS g\|_{W^{k+1,p}(\Om)}\leq
 C\|\om\|_{W^{k,p}(\Om)}\,.
\end{equation*}
Hence in this section we have proved the following:

\begin{theorem}\label{T.1}
The field $u$ given by~\eqref{defu} solves the problem~\eqref{problem}
and satisfies the estimate~\eqref{estimate}.
\end{theorem}

\section{Existence and bounds for the integral kernel}
\label{S.kernel}

In this section we shall show that the solution~$u$ is actually obtained by
integrating the field~$\om$ against an integral kernel. Specifically,
in terms of the function~$\ell$ defined in~\eqref{defl},
we aim to prove the following:

\begin{theorem}\label{T.2}
The solution~$u$ constructed in Theorem~\ref{T.1} (see
Equation~\eqref{defu}) is of the form
\[
u(x)=\int_\Om K_\Om(x,y)\, \om(y)\, dy
\]
for some matrix-valued kernel satisfying $|K_\Om(x,y)|\leq C\ell(y)/|x-y|^2$. 
\end{theorem}

Without loss
of generality, we can assume that $\om\in C^1(\BOm)$. We will
construct the kernel by treating separately the three summands
appearing in the decomposition
\[
u=v_1+v_2-\nabla\vp
\]
presented in Equations~\eqref{defv1}, \eqref{defv2} and~\eqref{defu}.
Since 
\[
v_1(x)=\int_\Om K_\Om^1(x,y)\, \om(y)\, dy
\]
with
\[
K_\Om^1(x,y)\om:=\frac{\om\times (x-y)}{4\pi|x-y|^3}\,,
\]
we will only need to show that $v_2$ and $\nabla\vp$ can also be
written in a similar fashion.

\subsection*{Step 1: The first term of the kernel of $v_2$} 

Let us start with $v_2$, which by
Equations~\eqref{defv2} and~\eqref{formulaf} can be written
as the limit as $\de\to0$ of the functions
\begin{align}\label{formv2}
v_2^\de(x)&:=\frac1{8\pi^2}\int_\Om\int_{\pd\Om\backslash B_\de(y)}\frac{z-y}{|z-y|^3}\times
\frac{x-y}{|x-y|^3}\,\big[\om\cdot\nu(z)-2T(\om\cdot\nu)(z)\big]\,
  d\si(z)\, dy\\
&=: \frac{V_1(x)-2V_2(x)}{8\pi^2}\,,\notag
\end{align}
with $T$ a pseudodifferential operator on~$\pd\Om$ of order $-1$ and
of the form~\eqref{defT}.

In this subsection we will work out the details for $V_1(x)$. Integrating by parts and using the fact that the divergence of~$\om$
is zero, one readily obtains that
\begin{equation}\label{I1I2}
V_1(x):=\int_\Om\int_{\pd\Om\backslash B_\de(y)}\frac{z-y}{|z-y|^3}\times
\frac{x-y}{|x-y|^3}\,\om\cdot\nu(z)\,
d\si(z)\, dy=I_1+I_2\,,
\end{equation}
where
\begin{align*}
I_1&:=\int_\Om\int_{\Om\backslash B_\de(y)}\pd_{z_j}\bigg(\frac{z-y}{|z-y|^3}\bigg)\times
\frac{x-y}{|x-y|^3}\,\om_j(z)\, dz\, dy\\
&\phantom{:} =-\int_\Om\int_{\Om\backslash B_\de(y)}
\frac{x-y}{|x-y|^3}\times \bigg(\frac{|z-y|^2I-3\,(z-y)\otimes(z-y)}{|z-y|^5}\cdot\om(z) \bigg)\, dz\, dy\,,\\[1mm]
I_2&:=\int_\Om\int_{\Om\cap\pd B_\de(y)}\frac{z-y}{|z-y|^3}\times
\frac{x-y}{|x-y|^3}\,\om\cdot\nu(z)\, d\si(z)\, dy\,,
\end{align*}
$\nu$ denotes the outward normal of~$\pd B_\de(y)$ and
the dot in the second lines denotes the multiplication of the vector
field~$\om(z)$ by a symmetric matrix. 

Let us begin by showing that $I_2$ tends to zero as $\de\to0$. For
this, observe that for $\om\in C^1(\BOm)$ one has
\begin{align*}
\int_{\Om\cap\pd B_\de(y)}\om(z)\cdot\nu(z)\,\frac{z-y}{|z-y|^3}\,
  d\si(z)&=\om(y)\cdot \int_{\Om\cap\pd B_\de(y)}\frac{\nu(z) \otimes(
           z-y)}{|z-y|^3}\,  d\si(z)+ O(\de) \,.
\end{align*}
To analyze the remaning integral, we shall begin by noticing that, in
terms of the variable $\Theta:=\frac1\de(z-y)$, the set $\Om\cap\pd B_\de(y)$ can be
described as
\[
\Om\cap\pd B_\de(y)=\big\{ \Theta\in S_{y,\de}\big\}\,,
\]
where $S_{y,\de}$ is a subset of the unit sphere~$\SS^2\subset\RR^3$ with
\[
S_{y,\de}=\SS^2\qquad\text{if } \dist(y,\pd\Om)>\de\,.
\]
Denoting by $d\Theta$ the canonical area form on the unit sphere, it is
clear that
\[
\int_{\Om\cap\pd B_\de(y)}\frac{\nu(z) \otimes(
           z-y)}{|z-y|^3}\,  d\si(z)=\int_{S_{y,\de}}(\Theta\otimes\Theta)\, d\Theta\,.
\]
Since the norm of the latter integral is obviously bounded by a constant
that does not depend on~$y$ or $\de$ and the integral is zero  when
$S_{y,\de}=\SS^2$, it then follows that
\begin{align*}
|I_2|\leq C \int_{\Om_\de}\frac{|\om(y)|}{|x-y|^2}\,
dy+C\de\leq C\de\,,
\end{align*}
with $\Om_\de:=\{ y\in\Om: \dist(y,\pd\Om)<\de\}$, thereby proving
that the boundary term vanishes in the limit $\de\to0$:
\begin{equation}\label{I20}
\lim_{\de\to0} I_2=0\,.
\end{equation}

Our next step will be to show that 
\[
|\widetilde I_1|\leq \frac{C\, \ell(z)}{|x-z|^2}
\]
with
\[
\widetilde I_1:=\int_{\Om\backslash B_\de(z)}
\frac{x-y}{|x-y|^3}\times
\frac{|z-y|^2I-3\,(z-y)\otimes(z-y)}{|z-y|^5}\, dy\,,
\]
where the cross product of a vector field with a matrix has
the obvious meaning. Let us take some small but fixed number
$a>0$. Clearly,
\begin{equation}\label{rho}
\bigg|\int_{\Om\backslash B_a(z)}
\frac{x-y}{|x-y|^3}\times
\frac{|z-y|^2I-3\,(z-y)\otimes(z-y)}{|z-y|^5}\, dy\bigg|\leq C\,,
\end{equation}
so it suffices to study the behavior of the integral
\begin{equation}\label{I3}
I_3:= \int_{\Om\cap(B_a(z)\backslash B_\de(z))}
\frac{x-y}{|x-y|^3}\times
\frac{|z-y|^2I-3\,(z-y)\otimes(z-y)}{|z-y|^5}\, dy\,.
\end{equation}
If the distance $\rho_z$ between the point $z$ and the
boundary is greater than~$a$, setting
\begin{equation}\label{eR}
e:=\frac{x-z}{R}\quad\text{and}\quad  R:= |x-z|
\end{equation}
one has
\begin{align*}
I_3&= \int_{B_a(z)\backslash B_\de(z)}
\frac{x-y}{|x-y|^3}\times
\frac{|z-y|^2I-3\,(z-y)\otimes(z-y)}{|z-y|^5}\, dy\\
&= \frac1{R^2}\int_{B_{a/R}\backslash B_{\de/R}}
\frac{e-w}{|e-w|^3}\times
\frac{|w|^2I-3\,w\otimes w}{|w|^5}\, dw\,,
\end{align*}
where we have defined $w:=(y-z)/R$. The integral is uniformly
convergent as $R\to0$ because the integrand is bound by
$C\,|w|^{-5}$ for large~$w$. For small $R_0$ and $\de<R\, R_0$ one has the
asymptotic behavior
\begin{align*}
  I_3&=\frac1{R^2} \int_{B_{R_0}\backslash  B_{\de/R}} \big(e+          O(w)\big)\times
                                            \frac{|w|^2I-3\,w\otimes                w}{|w|^5}\, dw \\
                                          & =  e\times \frac1{R^2}\int_{B_{R_0}\backslash
                                            B_{\de/R}} 
                                            \frac{|w|^2I-3\,w\otimes
                                            w}{|w|^5}\, dw + \frac1{R^2}\int_{B_{R_0}\backslash
                                            B_{\de/R}} O(|w|^{-2})\,
                                            dw\\
&=\frac1{R^2}\int_{B_{R_0}\backslash
                                            B_{\de/R}} O(|w|^{-2})\,
                                            dw<\frac C{R^2}
\end{align*}
with a constant independent of~$\de$. Here we have used that the first
integral in the second line vanishes and the average of the matrix
$w\otimes w$ over any sphere is $-\frac13 I$. Together
with~\eqref{rho}, this shows that
\[
|I_3|\leq \frac C{|x-z|^2}
\]
whenever the distance between $z$ and $\pd\Om$ is at least~$a$.

If $\rho_z:=\dist(z,\pd\Om)<a$, a similar argument yields the same
bound but with a constant that can grow as $\ell(z)$. In this case, one
begins by straightening out the boundary, so we locally identity the boundary with a vertical plane. This amounts to saying
that there is a diffeomorphism $\Phi_z$ of the ball $B_a$ such that 
\[
\Om\cap(B_a(z)\backslash B_\de(z))=\big\{ z+ \Phi_z(w):
w\in \cB_{a,\de,\rho_z}\big\}\,,
\]
with $\cB_{a,\de,r}$ being the intersection of the annulus of outer
radius~$a$ and inner radius~$\de$ with the half-space of points whose
third coordinate is at most~$r$:
\begin{equation}\label{cB}
\cB_{a,\de,r}:=\big\{ w\in\RR^3: \de<|w|<a,\; -r<w_3\big\}
\end{equation}
Furthermore, upon choosing a suitable orientation of the axes one can take
\[
\|\Phi_z-I\|_{C^k(B_a)}< Ca
\]
and one can assume without loss of generality that
\[
\Phi_z(0)=0\,.
\]
In this case one can write
\begin{align*}
I_3&= \int_{\cB_{a,\de,\rho_z}} \frac{x-z-\Phi_z(w)}{|x-z-\Phi_z(w)|^3}\times
\frac{|\Phi_z(w)|^2I-3\,\Phi_z(w)\otimes\Phi_z(w)}{|\Phi_z(w)|^5}\,
     |\det D\Phi_z(w)|\,  dw\\
&=\int_{\cB_{a,\de,\rho_z}} \frac{x-z-w}{|x-z-w|^3}\times
\frac{|w|^2I-3\,w\otimes w}{|w|^5}\,
      dw+ \int_{\cB_{a,\de,\rho_z}} \frac{x-z-w}{|x-z-w|^3}\times
\frac{O(1)}{|w|^2}\,
      dw\\
&=:I_{31}+I_{32}\,.
\end{align*}
One can introduce the variable $q:=
w/R$ and define~$e$ and~$R$ as in Equation~\eqref{eR}. The second
integral $I_{32}$ can be readily bounded by $C/R^2$, while one can
argue as before to obtain
\begin{align*}
I_{31}&=\frac1{R^2}\int_{\cB_{a/R,\de/R,\rho_z/R}} \frac{e-q}{|e-q|^3}\times
\frac{|q|^2I-3\,q\otimes q}{|q|^5} \,
     dq %+ \frac1{R^2}\int_{\cB_{a/R,\de/R,\rho_z/R}}O(|q|^{-2})\, dq
\\
&=e\times\frac1{R^2}\int_{\cB_{a/R,\de/R,\rho_z/R}} 
\frac{|q|^2I-3\,q\otimes q}{|q|^5} \,
     dq %+ \frac1{R^2}\int_{\cB_{a/R,\de/R,\rho_z/R}}O(|q|^{-2})\, dq
\,.
\end{align*}
The point now is that
\begin{equation}\label{I4}
I_4:=\int_{\cB_{a/R,\de/R,\rho_z/R}} \frac{|q|^2I-3\,q\otimes q}{|q|^5} \,
     dq
   \end{equation}
   is bounded by $C\log(2+1/\rho_z)$, with $C$ a constant that does not depend on~$\de$, which proves
the $\de$-independent bound
\begin{equation}\label{I3bound}
|I_{31}|<\frac {C\, \ell(z)}{|x-y|^2}\,.
\end{equation}
In order to see this, let us take spherical coordinates $(r,\te,\phi)$:
\[
q=:\frac rR\,\Theta(\te,\phi)\,,\qquad \Theta(\te,\phi):=(\sin\te\, \cos\phi, \; \sin\te\, \sin\phi,\; \cos\te)
\]
and use again the shorthand notation
\[
d\Theta:=\sin\te\, d\te\, d\phi
\]
for the surface measure on the unit sphere.  With
$\te_z\in(0,\frac\pi2)$ defined by
\[
\cos\te_z=\frac{\rho_z}a\,,
\]
the integral~\eqref{I4} can be written as
\begin{multline*}
I_4= \int_0^{2\pi}\int_0^{\te_z}\int_{\de}^{\frac{\rho_z}{\cos\te}}\frac{I
-3\,\Theta\otimes\Theta}r\,  dr\, \sin\te\, d\te\,  d\phi \\+ \int_0^{2\pi}\int_{\te_z}^\pi\int_{\de}^{a}\frac{I
-3\,\Theta\otimes\Theta}r\,  dr\, \sin\te\, d\te\,  d\phi% \\
% &= \int_\de^{a}\bigg(\int_{\SS^2}(I
% -3\,\Theta\otimes\Theta)\,  d\Theta\bigg)\,\frac{dr}r+ \int_0^{\te_z}\int_a^{a/\cos\te}\int_0^{2\pi}\frac{I
% -3\,\Theta\otimes\Theta}r\,  d\phi\,  dr\, \sin\te\, d\te\\
% &=\int_0^{\te_z}\int_a^{a/\cos\te}\int_0^{2\pi}\frac{I
% -3\,\Theta\otimes\Theta}r\,  d\phi\,  dr\, \sin\te\, d\te
\end{multline*}
Since
\[
\int_{\de}^{\frac{\rho_z}{\cos\te}}\frac{dr}r=\log\frac
a\de+\log\frac{\rho_z}{a\cos\te}\,,\qquad
\int_{\de}^{a}\frac{dr}r=\log\frac a\de\,,
\]
one obtains
\begin{align*}
|I_4|&=\bigg|\log\frac a\de\int_{\SS^2}(I
-3\,\Theta\otimes\Theta)\,d\Theta + \int_0^{2\pi}\int_0^{\te_z}(I
-3\,\Theta\otimes\Theta)\, \log\frac{\rho_z}{a\cos\te} \,\sin\te\,
d\te\, d\phi\bigg|\\
& = \bigg|\int_0^{2\pi}\int_0^{\te_z}(I
-3\,\Theta\otimes\Theta)\, \log\frac{\rho_z}{a\cos\te} \,\sin\te\,
d\te\, d\phi\bigg|\\
&\leq C(1+|\log\rho_z|)\leq C\ell(z)\,,
\end{align*}
where $C$ is independent of~$\de$ and to pass to the second line we have used again that the first integral after the equality
sign vanishes. This completes the proof of~\eqref{I3bound}.

The bounds for $I_{31}$ and $I_{32}$ then imply that the kernel
defined by setting
\[
\int_\Om K_1(x,z)\,\om(x)\, dz:= \frac1{8\pi^2}\lim_{\de\to 0}
 I_1\,,
\]
which is formally given by
\[
K_1(x,z)=\frac1{8\pi^2}\PV\int_\Om
\frac{x-y}{|x-y|^3}\times \frac{|z-y|^2I-3\,(z-y)\otimes(z-y)}{|z-y|^5}\, dy\,,
\]
is well-defined kernel bounded as
\[
|K_1(x,z)|\leq
\frac{C\ell(z)}{|x-z|^2}\,.
\]
Moreover, as $I_2\to0$ as $\de\to0$ by~\eqref{I20}, it then follows
from~\eqref{I1I2} that $K_1(x,z)$ satisfies
\begin{equation}\label{K2p}
\lim_{\de\to0} V_1(x)=\int_\Om K_1(x,z)\, \om(z)\, dz\,.
\end{equation}

\subsection*{Step 2: The second integral associated with~$v_2$}

To complete our analysis of~$v_2$, it remains to consider the term
in Equation~\eqref{formv2} having the quantity $T(\om\cdot\nu)$, which
we have called $V_2(x)$. Our goal is to show that
\[
\lim_{\de\to0}V_2(x)=\int_\Om K_2(x,y)\, \om(y)\, dy
\]
for some kernel that we will bound as $|K_2(x,y)|\leq C\ell(y)/|x-y|$.

Let us begin by using Proposition~\ref{P.PDO} to write
\[
T(\om\cdot\nu)=\tom\cdot\nu\,,
\]
where
\[
\tom:=\cE_T\om
\]
is the extension of $T(\om\cdot\nu)$ associated to the operator~$T$
defined by the integral~\eqref{tom}. To analyze the term $V_2(x)$ in
Equation~\eqref{formv2}, let us
write
\begin{align*}
V_2(x)=\int_{|y-z|>\de} \frac{z-y}{|z-y|^3}\times
\frac{x-y}{|x-y|^3}\, (\tom\cdot\nu)(z)\,
d\si(z)\, dy\,,
\end{align*}
where of course the integration variables $(y,z)$ range over
$\Om\times\pd\Om$. Integrating by parts one then finds
\[
V_2=I_5+I_6-I_7\,,
\]
where
\begin{align}
I_5&:=\int_{|y-z|>\de} \frac{x-y}{|x-y|^3}\times
     \frac{|z-y|^2I-3\,(z-y)\otimes(z-y)}{|z-y|^5}\cdot \tom(z)\, dy\,
     dz\,,\notag\\
I_6&:= \int_{|y-z|>\de} \frac{z-y}{|z-y|^3}\times
\frac{x-y}{|x-y|^3}\, \Div \tom(z)\,
 dy\, dz\,,\label{I6}\\
I_7&:= \int_{|y-z|=\de} \frac{z-y}{|z-y|^3}\times
\frac{x-y}{|x-y|^3}\, (\tom\cdot\nu)(z)\,
d\si(z)\, dy\,. \notag
\end{align}

Let us discuss the structure of these integrals. Arguing as in the case of~\eqref{I20}, it is easy to show that 
\[
\lim_{\de\to0} I_7=0\,.
\]
To study $I_5$ notice that, identifying $\cU$
with $\pd\Om\times(0, \rho_0)$ via the coordinates 
\[
z\mapsto(z',\rho_z)
\]
as in Section~\ref{S.extension}, one can write 
\[
\tom(z)=\nabla\rho(z)\int_{\pd\Om} K_{T}(z,w)\, \om^\perp(w,\rho_z)\, d\si(w)
\]
with a scalar kernel of the form~\eqref{kernelKT}. It stems from this
formula that
\begin{equation}\label{intI5}
\lim_{\de\to0} I_5=\int_\Om \cK(x,\bw)\, \om(\bw)\, d\bw\,,
\end{equation}
if we define $\cK(x,\bw)$ as the limit as $\de\to0$ of
\begin{multline}
\cK_\de:=\int_\de K_{T}(z,w)\,\bigg(\frac{x-y}{|x-y|^3}\times
     \frac{|z-y|^2 \nabla\rho(z)-3\,(z-y)\cdot \nabla\rho(z) \,
       (z-y)}{|z-y|^5}\bigg)\\
\otimes \nabla\rho(\bw)\, dy\,
     d\si(z')\label{cKde}
\end{multline}
provided that the latter exists. As before, we are denoting by $\bw$
the point in~$\cU$ of coordinates $(w,\rho_z)$, and the subscript
$\de$ in the integral will henceforth mean that we integrate over
points $(y,z')$ such that
\[
|y-z|>\de\quad \text{and}\quad |w-z'|>\de\,.
\]
Of course, $(y,z')\in\Om\times\pd\Om$, but in what follows we shall
not explicitly write the domain of integration (which will be
apparent) to keep the notation simple.

To analyze the behavior of this integral, in addition to exploiting the
diffeomorphism $\cU\to \pd\Om\times(0,\rho_0)$ we will take local normal
coordinates on $\pd\Om$, thereby identifying a point~$x$ in~$\cU$
(with $x'$ is certain open subset of $\pd\Om$) with the triple
$(X,\rho_x)$, with $X\equiv (X_1,X_2)$. Since we are using normal
coordinates on~$\pd\Om$, it is standard that, given two points $x,y$
of coordinates $(X,\rho_x)$, $(Y,\rho_y)$, one has
\begin{equation}\label{dist2}
|x-y|^2=(\rho_x-\rho_y)^2+|X-Y|^2+\hot\,,
\end{equation}
where in what follows we write $\hot$ to denote higher order terms.
Hence by Equation~\eqref{kernelKT} the kernel $K_T(z,w)$ can be
written in these coordinates as
\begin{equation}\label{Kt4pi}
K_T(z,w)=\frac1{4\pi}\frac{q(\zeta)}{(\rho_z^2+|\zeta|^2)^{3/2}}+\hot\,,
\end{equation}
where $\zeta:=Z-W$ and $q(\zeta):=q_{ij}\zeta_i\zeta_j$ is a quadratic form (here we have used that
$(z'-w)\cdot \nu(z')$ vanishes to second order). Let us set $\ty:=y-z$ and write
\[
|x-y|= |x-z-\ty|=|x-\bw-\zeta-\ty|+\hot\,,
\]
where we have identified $\zeta$ with the point of coordinates
$(0,\zeta)$ and made use of~\eqref{dist2}. This permits us to write
\begin{align*}
\cK_\de=\int_\de
         \frac{q(\zeta)}{(\rho_z^2+|\zeta|^2)^{3/2}}\frac{x-\bw-\zeta-\ty}{|x-\bw-\zeta-\ty|^3}\times
         \frac{|\ty|^2e_3-3\ty \ty_3}{4\pi|\ty|^5}\otimes e_3\,
         d\ty\, d\zeta+\hot\,,
\end{align*}
where the subscript~$\de$ again means that
\[
|\ty|>\de\,,\qquad |x-\bw-\zeta-\ty|>\de\,.
\]

Let us now write
\[
x-\bw=:R\, e\,,\qquad R:=|x-\bw|
\]
and observe that, for small~$R$, one can rescale the variables as
\[
\ty':=\frac{\ty}R\,,\qquad \zeta':=\frac\zeta R\,,\qquad \rho_z':=\frac{\rho_z}R
\]
and write the above integral as
\[
\cK_\de=\frac1R\int_{\de/R}\frac{q(\zeta')}{(\rho_z'^2+|\zeta'|^2)^{3/2}}\frac{e-\zeta'-\ty'}{|e-\zeta'-\ty'|^3}\times
         \frac{|\ty'|^2e_3-3\ty' \ty_3'}{4\pi|\ty'|^5}\otimes e_3\,
         d\ty'\, d\zeta'+\hot\,,
\]
where of course the rescaled domains of integration become unbounded
in the limit~$R\to0$. 

Our goal now is to show that
\begin{equation}\label{I8}
I_8:= \int_{\de/R}\frac{q(\zeta')}{(\rho_z'^2+|\zeta'|^2)^{3/2}}\frac{e-\zeta'-\ty'}{|e-\zeta'-\ty'|^3}\times
         \frac{|\ty'|^2e_3-3\ty' \ty_3'}{|\ty'|^5}\,
         d\ty'\, d\zeta'
       \end{equation}
       is bounded as
\[
|I_8|\leq C\,\ell(z)
\]
uniformly as $\de\to0$. Observe that this will show that
\[
|\cK_\de|\leq \frac {C\,\ell(z)}R\,,
\]
which will in turn ensure the existence of the kernel $\cK(x,\bw)$ with a
bound
\begin{equation}\label{boundcK}
|\cK(x,\bw)|\leq\frac {C\,\ell(z)}{|x-\bw|}\,.
\end{equation}

To prove the bound for $I_8$ it suffices to analyze the behavior of the
integrand at the points where it can be not uniformly in $L^1\loc$
(that is, in a neighborhood of the regions $e-\zeta'-\ty'=0$ and
$\ty=0$ ). Let us start with this first case. Around $e-\zeta'-\ty'=0$ one can write
\[
e-\ty'=(B,b)\,,
\]
with $B\equiv (B_1,B_2)\in\RR^2$. It is clear from that it is enough to consider
values of $\ty$ that are close to zero, say with $|\ty|<\frac14$. It
then follows that
\begin{equation}\label{B12}
|B|^2+b^2\geq \frac12\,,
\end{equation}
so that the integral in~$\zeta'$ can be written in terms of the
variable
\[
E:=\zeta'-B
\]
as
\begin{align*}
I_9&:=\int_{\de/R}\frac{q(\zeta')}{(\rho_z'^2+|\zeta'|^2)^{3/2}}\frac{e-\zeta'-\ty'}{|e-\zeta'-\ty'|^3}\,
     d\zeta'\\
&= \int_{D_{a/R}\backslash D_{\de/R}} \frac{q(\zeta')}{(\rho_z'^2+|\zeta'|^2)^{3/2}}\frac{e-\zeta'-\ty'}{|e-\zeta'-\ty'|^3}\,
     d\zeta'+O(1)\\
&=\int_{D_{a/R}\backslash D_{\de/R}} \frac{q(E+B)}{(\rho_z'^2+|E+B|^2)^{3/2}}\frac{(E,b)}{(b^2+|E|^2)^{3/2}}\,
     dE+O(1)\,.
\end{align*}
The possible problem can arise as $b\to0$. Since in this case $B$ is
bounded away from zero by~\eqref{B12}, however, one can write $I_9$ as
\begin{align*}
I_9&=\frac{q(B)}{(\rho_z'^2+|B|^2)^{3/2}}\int_{D_{a/R}\backslash D_{\de/R}} \frac{(E,b)}{(b^2+|E|^2)^{3/2}}\,
     dE+O(1) \\
&=\frac{q(B)}{(\rho_z'^2+|B|^2)^{3/2}}\,\bigg(\int_{D_{a/R}\backslash D_{\de/R}} \frac{E \, dE}{(b^2+|E|^2)^{3/2}}\,, \; \int_{D_{a/R}\backslash D_{\de/R}} \frac{b \, dE}{(b^2+|E|^2)^{3/2}}\bigg) + O(1)\,.
\end{align*}
The first integral vanishes by parity, so one can estimate $I_9$ as
\begin{align*}
|I_9|&\leq C\int_{\RR^2}\frac{b \,
       dE}{(b^2+|E|^2)^{3/2}}+O(1)=C\int_{\RR^2}\frac{dE}{(1+|E|^2)^{3/2}}+O(1)\leq
       C\,.
\end{align*}
It remains now to consider the behavior of the integral with respect
to~$\ty$ around $\ty=0$. This can be handled exactly
as in the case of~$I_4$, which yields to a bound of the form
$C\log(2+\rho_z^{-1})$. Combining both results one immediately obtains that
\[
|I_8|\leq C\,\ell(z)\,,
\]
thereby establishing~\eqref{boundcK}.

\subsection*{Step 3: The remaining integrals}

The rest of the proof of Theorem~\ref{T.2} follows by repeatedly
applying the ideas that we have used above. Let us sketch the remaning
steps. 

To complete our treatment of~$v_2$, one has to show that, just as in~\eqref{intI5},
\[
\lim_{\de\to0}I_6=\int_\Om \cK'(x,\bw)\, \om(\bw)\, d\bw
\]
for some suitable integral kernel, where $I_6$ is given by~\eqref{I6}. In view of the formula for the
divergence of~$\tom$ given in
Proposition~\ref{P.PDO}, this kernel arises as the limit as
$\de\to0$ of the integral
\[
\cK'_\de:=\int_\de \bigg(\frac{x-y}{|x-y|^3}\times
     \frac{z-y}{|z-y|^3}\bigg) \otimes K_{T,\mathrm{div}}(z,w)\, dy\,
     d\si(z')\,,
\]
where as before the subscript~$\de$ means that one only integrates
over $(y,z')$ with 
\[
|y-z|>\de\,,\qquad |z'-w|>\de\,.
\]
Note that the kernel $K_{T,\mathrm{div}}$, which was introduced
in~\eqref{Kdiv}, diverges as the inverse square of the
distance. Setting, in local coordinates $(Z,\rho_z)$ as above, 
\[
\zeta:=Z-W\,,\qquad \ty:=y-z\,,\qquad x-\bw=:R\, e
\]
with $R:=|x-\bw|$, and rescaling the integral as before, one finds
that
\[
\cK'_\de=\frac1R I_{10}+\hot\,,
\]
where $I_{10}$ is a certain integral with respect to rescaled
variables $(\ty',\zeta')$ that is shown to be bounded by $C\,\ell(z)$ by
tediously repeating the steps taken before, with only minor
modifications. This completes the proof of the existence of a kernel
bounded as
\[
|K^2(x,z)|\leq \frac{C\, \ell(z)}{|x-z|^2}
\]
and such that
\[
\lim_{\de\to0} V_2(x)=\int_\Om K^2(x,z)\, \om(z)\, dz\,.
\]

The analysis of $\nabla\vp$ is similar. Since
\[
\nabla\vp(x)=-\int_{\pd\Om}\frac{x-y}{4\pi|x-y|^3}g(y)\, d\si(y)
\]
with~$g$ given in terms of~$v\cdot\nu$ by~\eqref{formulag}. In order to prove that one can
write
\[
\nabla\vp(x)=\int_\Om K^3(x,z)\, \om(z)\, dz
\]
with a kernel bounded as
\[
|K^3(x,z)|\leq\frac {C\,\ell(z)}{|x-z|^2}
\]
one starts off by writing 
\begin{align*}
\nabla\vp(x)&=-\int_{\pd\Om}\frac{x-y}{2\pi|x-y|^3}v\cdot\nu(y)\,
              d\si(y)-\int_{\pd\Om}\frac{x-y}{\pi|x-y|^3}\widetilde T(v\cdot\nu)(y)\,
              d\si(y)\\
&=:-\frac{J_1+2J_2}{2\pi}\,,
\end{align*}
where $\widetilde T$ is an operator of the form~\eqref{defT}. Integrating by
parts and arguing as before,
one can readily infer that
\begin{align*}
J_1&=-\lim_{\de\to0}\int_{|y-x|>\de}
     v(y)\,\frac{|x-y|^2I-3(x-y)\otimes (x-y)}{|x-y|^5}\, dy\\
&=-\frac1{8\pi^2}\lim_{\de\to0}\bigg(\int_{|y-x|>\de}
     V_1(y)\,\frac{|x-y|^2I-3(x-y)\otimes (x-y)}{|x-y|^5}\, dy \\
&\qquad\qquad\qquad \qquad\qquad+2\int_{|y-x|>\de}
     V_2(y)\,\frac{|x-y|^2I-3(x-y)\otimes (x-y)}{|x-y|^5}\, dy\bigg)\\
&=:-\frac1{8\pi^2}\lim_{\de\to0}(J_{11}+2J_{12})\,.
\end{align*}
Given the expression of $V_1$, it turns out that one can integrate by
parts to write
\begin{multline*}
\lim_{\de\to0}J_{11}=\lim_{\de\to0}\int_\de\om(z)\frac{|z-w|^2I-3(z-w)\otimes(z-w)}{|z-w|^5}\times
\frac{y-w}{|y-w|^3}\\
\cdot
\frac{|x-y|^2I-3(x-y)\otimes(x-y)}{|x-y|^5}\, dw\, dy\, dz\,.
\end{multline*}
Hence essentially the same reasoning that we used with $I_3$ allows us to show that
\[
 J_{11}=\int K^{31}(x,z)\, \om(z)\, dz
\]
with a kernel satisfying the bound
\[
|K^{31}(x,z)|\leq \frac{C\, \ell(z)}{|x-z|^2}
\]
that is given by the limit as $\de\to0$ of
the integral
\[
\lim_{\de\to0}\int_\de \frac{|z-w|^2I-3(z-w)\otimes(z-w)}{|z-w|^5}\times \frac{y-w}{|y-w|^3}\cdot
\frac{|x-y|^2I-3(x-y)\otimes(x-y)}{|x-y|^5}\, dw\, dy\,.
\]
Likewise, arguing exactly as in the analysis of $V_2$ one can show
that in the limit $\de\to0$ (which will not be written explicitly
for the ease of notation) one has
\begin{align*}
 J_{12}&= \int_\de \frac{z-w}{|z-w|^3}\times\frac{y-w}{|y-w|^3}\cdot
         \frac{|x-y|^2I-3(x-y)\otimes (x-y)}{|x-y|^5}\tom\cdot\nu(z)\,
         d\si(z)\, dw\, dy\\
&= \int_\de \tom(z)\cdot \frac{|z-w|^2I-3(z-w)\otimes(z-w)}{|z-w|^5}\times\frac{y-w}{|y-w|^3}\\
&\qquad\qquad \qquad\qquad\qquad \qquad \cdot
         \frac{|x-y|^2I-3(x-y)\otimes (x-y)}{|x-y|^5}\tom\cdot\nu(z)\,
         dz\, dw\, dy\\
&-\int_\de \frac{z-w}{|z-w|^3}\times\frac{y-w}{|y-w|^3}\cdot
         \frac{|x-y|^2I-3(x-y)\otimes (x-y)}{|x-y|^5}\Div\tom(z)\,
         dz\, dw\, dy\,.
\end{align*}
Writing $\tom$ and $\Div\tom$ in terms of $\om$ using
Lemma~\ref{P.PDO} as above one finds that
\[
\lim_{\de\to0} J_{12}=\int K^{32}(x,q)\, \om(q)\, dq\,,
\]
where the kernel $K^{32}$ is bounded as
\[
|K^{32}(x,q)|\leq \frac{C\, \ell(q)}{|x-q|}\,.
\]
The treatment of $J_2$ can be accomplished using a completely
analogous reasoning, yielding
\[
 J_2=\int K^{33}(x,z)\,\om(z)\, dz
\]
with 
\[
|K^{32}(x,z)|\leq \frac {C\, \ell(z)}{|x-z|}\,.
\]
The details, which are tedious but now straightforward, are omitted. This completes the proof
of Theorem~\ref{T.2}.

\section{Uniqueness and an application to the full div-curl system}
\label{S.uniqueness}

As a last step in the proof of Theorem~\ref{T.main}, in this easy section we will
consider the uniqueness of the solution. For completeness, we will do
so in the context of the general div-curl system
\begin{align}\label{system}
\curl v=\om\,,\qquad \Div v=f\,,\qquad v\cdot \nu=g
\end{align}
in~$\Om$. Here $\om$ is a divergence-free field satisfying~\eqref{hypothesis}
and the functions $f$ and $g$ satisfy the well known compatibility conditions
\begin{equation}\label{condfg}
\int_\Om f\, dx=\int_{\pd\Om} g\, d\si\,.
\end{equation}

\begin{proposition}\label{P.system}
The system~\eqref{system}, with $\om$, $f$ and $g$ 
as above, admits a solution~$v$, which is bounded as
\[
\|v\|_{W^{k+1,p}(\Om)}\leq C(\|\om\|_{W^{k,p}(\Om)}+ \|f\|_{W^{k,p}(\Om)}+\|g\|_{W^{k+1-\frac1p,p}(\pd\Om)})\,.
\]
Furthermore, the solution is unique modulo the addition of a harmonic
field tangent to the boundary, and the dimension of this linear space
equals the genus of~$\pd\Om$.
\end{proposition}

\begin{proof}
Uniqueness is immediate: if $v$ and $v'$ are two solutions to the
problem, their difference $w:=v-v'$ satisfies
\begin{equation}\label{harm}
\curl w=0\,,\qquad \Div w=0\,,\qquad w\cdot \nu=0\,,
\end{equation}
so it is a harmonic field on~$\Om$ tangent to the boundary. The
dimension of the linear space of solutions to~\eqref{harm} is known to
be given by the genus of~$\pd\Om$ by Hodge theory.

To prove the existence of a solution, let $\phi$ be the only solution
to the problem
\[
\De\phi=f\,,\qquad \pd_\nu\phi=g\,,\qquad \int_\Om f\, dx=0
\]
in~$\Om$, which is granted to exist by the hypothesis~\eqref{condfg}
and satisfies
\[
\|\phi\|_{W^{k+2,p}(\Om)}\leq C(\|f\|_{W^{k,p}(\Om)}+\|g\|_{W^{k+1-\frac1p,p}(\pd\Om)})\,.
\]
The field $u:=v-\nabla\phi$ then satisfies
\[
\curl u=\om\,,\qquad \Div u=0\,,\qquad u\cdot\nu=0\,,
\]
so Theorem~\eqref{T.1} ensures that there is a solution satisfying
\[
\|u\|_{W^{k+1,p}(\Om)}\leq C\|\om\|_{W^{k,p}(\Om)}\,.
\]
The statement then follows.
\end{proof}

\section*{Acknowledgments}

The authors are supported by the ERC Starting Grants~633152 (A.E.\ and
M.A.G.F.) and~335079
(D.P.S.) and by an ICMAT-SO scholarship (M.A.G.F.) from the Spanish
Ministry of Economy. This work is supported in part by the Spanish
Ministry of Economy under the
ICMAT--Severo Ochoa grant
SEV-2015-0554.

\bibliographystyle{amsplain}

\end{document}